\documentclass[11pt,reqno]{amsart}
\usepackage[T1]{fontenc}

\usepackage {amsmath, amsfonts, amssymb, amscd, amsthm, eucal}%,a4wide

\usepackage{amsbsy}
\usepackage[all]{xy}
\SelectTips{cm}{}
\SilentMatrices
\usepackage[dvipsnames,usenames]{color}
\theoremstyle{plain} {
  \newtheorem{thm}{Theorem}[section]
  \newtheorem{cor}[thm]{Corollary}
  \newtheorem{lem}[thm]{Lemma}
  \newtheorem{prop}[thm]{Proposition}
  
  \newtheorem{sssection}[thm]{}
}

\theoremstyle{definition}
{
  \newtheorem{defn}[thm]{Definition}
  
  \newtheorem{example}[thm]{Example}
  
}

\theoremstyle{remark}
{
  
}

\renewcommand{\subsubsection}{\sssection\rm}

\newcommand{\pt}{pt}

\newcommand{\rk}{\mathrm{rk}}

\DeclareMathOperator{\Th}{Th}
\newcommand{\MGL}{\mathbf {MGL}}
\newcommand{\SmOp}{\mathcal Sm\mathcal Op}
\newcommand{\Sm}{\mathcal Sm}

\newcommand{\Aff}{\mathbf {A}}

\newcommand{\ZZ}{\mathbb{Z}}

\newcommand{\colim}{\operatorname{colim}}
\newcommand \xra {\xrightarrow }

\newcommand \lra {\longrightarrow }
\newcommand \hra {\hookrightarrow }

\newcommand \xla {\xleftarrow }

\newcommand{\BO}{\mathbf{BO}}
\newcommand{\MSp}{\mathbf{MSp}}
\newcommand{\MSL}{\mathbf{MSL}}
\newcommand{\KO}{\mathbf{KO}}
\newcommand{\KSp}{\mathbf{KSp}}
\newcommand{\HH}{\mathsf{H}}
\newcommand{\GG}{\mathbb{G}}
\newcommand{\OO}{\mathcal{O}}
\newcommand{\thy}[1]{\mathrm{#1}}
\DeclareMathOperator{\thom}{\mathnormal{th}}
\newcommand{\shf}{\mathcal}

\DeclareMathOperator{\Pf}{Pf}
\newcommand{\into}{\hra}

\newcommand{\hh}{\mathsf{h}}
\newcommand{\signtrans}{\text{\textup{\textsf{sign.trans}}}}
\newcommand{\unsigntrans}{\text{\textup{\textsf{unsign.trans}}}}
\newcommand{\parens}[1]{\textup{(}#1\textup{)}}
\newcommand{\homog}{\text{\textit{hom}}}

\begin{document}

\title{On the relation of symplectic algebraic cobordism to hermitian $K$-theory}

\author{I.~Panin}
\address{Steklov Institute of Mathematics at St.~Petersburg, Russia}
\author{C.~Walter}
 \address{Laboratoire J.-A.\ Dieudonn\'e (UMR 6621 du CNRS)\\
 D\'epartement de math\'ematiques\\
 Universit\'e de Nice -- Sophia Antipolis\\ 06108 Nice Cedex 02\\
 France}

%\date{September 20, 2010}
\thanks{The first author gratefully acknowledge excellent working conditions and support provided by
Laboratoire J.-A. Dieudonn\'{e}, UMR 6621 du CNRS, Universite de Nice - Sophia-Antipolis,
and by the RCN Frontier Research Group Project no. 250399 Motivic Hopf equations at University of Oslo.
}

\begin{abstract}
We reconstruct hermitian $K$-theory via algebraic symplectic cobordism.
In the motivic stable homotopy category $SH(S)$
there is a unique morphism
$\varphi\colon \MSp \to \BO$
of commutative ring $T$-spectra
which sends the Thom class
$th^{\MSp}$
to the Thom class
$th^{\BO}$.
Using $\varphi$ we construct an isomorphism of bigraded ring
cohomology theories on the category
$\SmOp/S$
$$
\bar\varphi\colon \MSp^{\ast,\ast}(X,U)
\otimes_{\MSp^{4\ast, 2\ast}(\pt)} \BO^{4\ast, 2\ast}(\pt)
\cong \BO^{\ast,\ast}(X,U).
$$
The result is an algebraic version of the theorem of Conner and Floyd
reconstructing real $K$-theory using symplectic cobordism.
Rewriting the bigrading as $\MSp^{p,q} = \MSp^{[q]}_{2q-p}$,
we have an isomorphism
\[
\bar\varphi\colon \MSp^{[*]}_{*}(X,U)
\otimes_{\MSp^{[2*]}_{0}(\pt)} KO^{[2*]}_0(\pt)
\cong KO^{[*]}_{*}(X,U),
\]
where the $KO^{[n]}_{i}(X,U)$ are Schlichting's hermitian $K$-theory groups.
\end{abstract}

\maketitle

\section{A motivic version of a theorem by Conner and Floyd}
\label{Introduction}
Our main result relates symplectic
algebraic cobordism
to
hermitian
$K$-theory.
It is an algebraic version of the theorem of Conner and Floyd \cite[Theorem 10.2]{Conner:1966uk}
reconstructing real $K$-theory using symplectic cobordism.  The algebraic version of the
reconstruction of complex $K$-theory using unitary cobordism was done in \cite{Panin:2009fp}.

In \cite{Panin:2010aa} the current authors constructed a commutative ring $T$-spectrum
$\BO$
representing hermitian $K$-theory in the stable homotopy
category $SH(S)$ for any regular noetherian separated base scheme $S$ of finite Krull dimension without
residue fields of characteristic $2$.  (These restrictions allowed us to use particularly strong
results of Marco Schlichting \cite{Schlichting:2006aa}.  We leave it to the expert(s) in negative
hermitian $K$-theory to weaken them.)  It has a standard family of
Thom classes for special linear vector bundles and hence for symplectic bundles.
%In view of \cite{Panin:2010ab}
The symplectic Thom classes can all be derived from a single class
$\thom^{\BO} \in \BO^{4,2}(\Th \shf U_{HP^{\infty}}) = \BO^{4,2}(\MSp_{2})$, the symplectic
Thom orientation.

In \cite{Panin:2010ab} we constructed the commutative ring $T$-spectrum $\MSp$
of algebraic symplectic cobordism.  It is a commutative monoid in the model category
of symmetric $T^{\wedge 2}$-spectra, just as $\MSL$ and Voevodsky's $\MGL$ are commutative
monoids in the model category of symmetric $T$-spectra.  The canonical map
$\Sigma_{T}^{\infty}\MSp_{2}(-2) \to \MSp$ gives the symplectic Thom orientation
$\thom^{\MSp} \in \MSp^{4,2}(\MSp_{2})$.  It is the universal symplectically oriented
commutative ring $T$-spectrum.

Therefore there is a unique morphism $\varphi\colon \MSp \to \BO$ of commutative monoids
in $SH(S)$ with $\varphi(\thom^{\MSp}) = \thom^{\BO}$.  Our main result is the following theorem.
Our notation is that for a motivic space $Y$
and a bigraded cohomology theory
we write
$A^{*,*}(Y) = \bigoplus_{p,q \in \ZZ} A^{p,q}(Y)$ and
$A^{4*,2*}(Y) = \bigoplus_{i \in \ZZ} A^{4i,2i}(Y)$.
%$\MSp^{*,*}(Y) = \bigoplus_{p,q \in \ZZ} \MSp^{p,q}(Y)$ and
%$\MSp^{4*,2*}(Y) = \bigoplus_{i \in \ZZ} \MSp^{4i,2i}(Y)$
%%%
%%\begin{align*}
%%\MSp^{*,*}(Y) & = \bigoplus_{p,q \in \ZZ} \MSp^{p,q}(Y), &
%%\MSp^{4*,2*}(Y) & = \bigoplus_{i \in \ZZ} \MSp^{4i,2i}(Y),
%%\end{align*}
%%%
%and similarly for $\BO^{*,*}$ and $\BO^{4*,2*}$.
%The rings $\MSp^{*,*}(Y)$ and $\BO^{*,*}(Y)$ satisfy a certain bigraded-commutativity relation,
%and the subrings $\MSp^{4*,2*}(Y)$ and $\BO^{4*,2*}$ are central.
A motivic space $Y$ is \emph{small} if
$Hom_{SH(S)}(\Sigma^\infty_{T} Y,{-})$
commutes with arbitrary coproducts.

\begin{thm}
\label{T:main.CF}
Let $S$ be a regular noetherian separated scheme of finite Krull dimension with
$\frac 12 \in \Gamma(S,\OO_{S})$.
For all small pointed motivic spaces $Y$ over $S$ the map
\[
\bar\varphi\colon \MSp^{\ast,\ast}(Y)
\otimes_{\MSp^{4\ast, 2\ast}(\pt)} \BO^{4\ast, 2\ast}(\pt)
\to \BO^{\ast,\ast}(Y).
\]
induced by $\varphi$ is an isomorphism.
\end{thm}

This has as a consequence the result mentioned in the abstract.  For a pair $(X,U)$ consisting of a
smooth $S$-scheme of finite type $X$ and an open subscheme $U$, there is a
quotient pointed motivic space $X_{+}/U_{+}$.
We define $\MSp^{*,*}(X,U) = \MSp^{*,*}(X_{+}/U_{+})$ and $\BO^{*,*}(X,U) = \BO^{*,*} (X_{+}/U_{+})$.
There are natural isomorphisms $\BO^{p,q}(X,U) = KO_{2q-p}^{[q]}(X,U)$ with the hermitian
$K$-theory of $X$ with supports in $X-U$ as defined by Schlichting \cite{Schlichting:2010uq}.
The weight $q$ is the degree of the shift in the duality used for the symmetric bilinear forms
on the chain complexes of vector bundles.

For a field $k$ of characteristic not $2$ the ring $\BO^{4*,2*}(k)$ is not large.  For all $i$ one has
$\BO^{8i,4i}(k) \cong GW(k)$ and $\BO^{8i+4,4i+2}(k) \cong \ZZ$.  All members of $\BO^{0,0}(k)$
therefore come from composing endomorphisms in $SH(k)$ of the sphere $T$-spectrum
$\boldsymbol 1 = \Sigma_{T}^{\infty}\pt_{+}$ with the unit $e \colon \boldsymbol 1 \to \BO$ of
the monoid.  (See Morel \cite[Theorem 4.36]{Morel:2006aa} and Cazanave \cite{Cazanave:2010aa}
for calculations of the endomorphisms of the sphere $T$-spectrum.)  Consequently
$\varphi^{0,0} \colon \MSp^{0,0}(k) \to \BO^{0,0}(k)$ is surjective.  We do not know what happens
in other bidegrees.

This is the fourth in a series of papers about symplectically oriented motivic cohomology theories.
All depend on the quaternionic projective bundle theorem proven in the first paper \cite{Panin:2010fk}.
%The result in this
%paper is a consequence of the quaternionic projective bundle theorem.

%We refer to
%\cite[Appendix]{Panin:2009aa}
%for the basic terminology, notation, constructions, definitions,
%results.

\section{Preliminaries}

Let $S$ be a Noetherian separated scheme of finite Krull dimension.
We will be dealing with hermitian $K$-theory, and we prefer avoiding the subtleties of
negative $K$-theory, so we will assume as we did in \cite{Panin:2010aa} that $S$ is regular and that
$\frac 12 \in \Gamma(S,\OO_{S})$.
Let $\Sm/S$ be the category of smooth $S$-schemes of finite type.
Let $\SmOp/S$ be the category whose objects are pairs $(X,U)$ with $X \in \Sm/S$ and $U \subset X$
an open subscheme and whose arrows $f \colon (X,U) \to (X',U')$ are morphisms
$f \colon X \to X'$ of $S$-schemes with $f(U) \subset U'$.  Note that all $X$ in $\Sm/S$ have an ample
family of line bundles.

%One may think of $S$ being the spectrum of a field or the integers.
A \emph{motivic space
over} $S$ is a simplicial presheaf on $\Sm/S$.  We will often write $\pt$ for the base scheme regarded
as a motivic space over itself.
%$X\colon (\Sm/S)^{op} \to \mathbf{sSet}$.
%(see \cite[Appendix]{Panin:2009aa}).
Inverting the motivic weak equivalences in the category of pointed motivic spaces gives
the pointed motivic unstable homotopy category $H_{\bullet}(S)$.

Let $T = \Aff^{1}/(\Aff^{1}-0)$ be the Morel-Voevodsky object.
A \emph{$T$-spectrum} $M$ is a sequence of pointed motivic spaces $(M_{0},M_{1},M_{2},\dots)$
equipped with structural maps $\sigma_{n} \colon M_{n} \wedge T \to M_{n+1}$.
Inverting the stable motivic weak equivalences gives the motivic stable homotopy category $SH(S)$.
A pointed motivic space $X$ has a $T$-suspension spectrum $\Sigma_{T}^{\infty}X$.  For any
$T$-spectrum $M$ there are canonical maps of spectra
\begin{equation}
\label{E:canonical.map}
u_{n} \colon \Sigma_{T}^{\infty}M_{n}(-n) \to M.
\end{equation}

%We will work mostly with the homotopy categories.

%The functor represented by a smooth $S$-scheme is an unpointed motivic space.
%The functor represented by the base scheme $S$ will often be written $\pt$.
%Given an unpointed motivic space $X$ there is a pointed motivic space $X_{+} = X \sqcup \pt$
%pointed by the inclusion ${+} \colon \pt \to X \sqcup \pt$.
%
%Both homotopy categories are closed symmetric monoidal for the smash product $\wedge$.

%The category of motivic spaces over $S$ is denoted
%$\mathbf{M}(S)$. This definition of a motivic space is different from the one considered
%by Morel and Voevodsky in \cite{Morel:1999ab} -- they consider only those simplicial presheaves
%which are sheaves in the Nisnevich topology on $\Sm/S$. With our definition the
%Schlichting's hermitian $\mathrm{KO}$-theory functor obtained by using perfect complexes of
%big vector bundles is a motivic space
%on the nose
%\cite[Defn.5.6]{Panin:2010aa}. It is not a simplicial Nisnevich sheaf. This is why we prefer to work with the
%above notion of ``space''.
%
%We write
%$H_\bullet (S)$
%for the pointed motivic homotopy category
%and
%$SH(S)$
%for
%the stable motivic homotopy category over $S$ as constructed in
%\cite[A.3.9, A.5.6]{Panin:2009aa}.
%By
%\cite[A.3.11 resp.~A.5.6]{Panin:2009aa}
%there are canonical equivalences to
%$H_\bullet(S)$
%of
%\cite{Morel:1999ab}
%resp.
%$SH(S)$
%of
%\cite{Voevodsky:1998kx}.
Both
$H_\bullet(S)$
and
$SH(S)$
are equipped with
closed symmetric monoidal structures, and
%the $T$-suspension spectrum functor
$\Sigma^\infty_{T} \colon H_\bullet(S)\to SH(S)$
is a strict symmetric monoidal functor.
%Here
%$T = \Aff^{1}/(\Aff^{1} - 0)$
%is the Morel-Voevodsky space.
The symmetric monoidal structure $(\wedge,\boldsymbol{1}_S = \Sigma^\infty_{T}\pt_+)$
on the homotopy category
$SH(S)$ can be constructed on the model category level
using
%$\MSS(S)$
symmetric $T$-spectra.
%It satisfies the properties required by
%Theorem 5.6 of Voevodsky congress talk
%\cite{Voevodsky:1998kx}.
%From now on we will usually omit the superscript $(-)$.

Any $T$-spectrum $A$ defines a cohomology theory on the category
of pointed motivic spaces. Namely, for a pointed space $(X,x)$ one sets
$A^{p,q}(X,x)=Hom_{H_\bullet (S)}(\Sigma^\infty_{T}(X,x), \Sigma^{p,q}(A))$
and
$A^{\ast,\ast}(X,x)= \bigoplus_{p,q\in \ZZ} A^{p,q}(X,x)$.
We write (somewhat inconsistently)
\[
A^{4*,2*}(X,x) = \bigoplus_{i \in \ZZ} A^{4i,2i}(X,x).
\]
%A cohomology theory on the category of unpointed spaces is defined as follows.
For an unpointed space
$X$ we set
$A^{p,q}(X)=A^{p,q}(X_+,+)$, with
$A^{\ast,\ast}(X)$ and $A^{4*,2*}(X)$ defined accordingly.  We will not always write
the pointings explicitly.

Each
$Y \in \Sm/S$
defines an unpointed motivic space which is constant in the simplicial direction
%taking an smooth $S$-scheme $U$ to
$Hom_{\Sm/S}({-},Y)$.
So we regard smooth $S$-schemes as motivic spaces
and set
$A^{p,q}(Y)=A^{p,q}(Y_+,+)$.  Given a monomorphism $U \hra Y$ of smooth $S$-schemes, we write
$A^{p,q}(Y,U) = A^{p,q}(Y_{+}/U_{+},U_{+}/U_{+})$.

%Given a $T$-spectrum $E$ we will reduce the double grading on the cohomology
%theory
%$E^{\ast,\ast}$
%to a grading. Namely, set
%$E^m = \oplus_{m = p-2q}E^{p,q}$
%and
%$E^{\ast}=\oplus_{m}E^m$.
%{\it We often will write}
%$E^{\ast}(k)$
%{\it for}
%$E^{\ast}(\mathrm{Spec}(k))$
%{\it below in this text}.

%A ring $T$-spectrum is a monoid
%$(E,\mu,e)$
%in
%$(SH(S),\wedge, \boldsymbol{1}_S)$.
A \emph{commutative ring $T$-spectrum} is a commutative monoid
$(A,\mu,e)$
in
$(SH(S),\wedge, 1)$.

%The cohomology theory $E^{\ast,\ast}$ defined by a
%ring $T$-spectrum is a ring cohomology theory.
The cohomology theory $A^{\ast,\ast}$ defined by a commutative
ring $T$-spectrum is a ring cohomology theory
satisfying a certain bigraded commutativity condition described by Morel.
Namely, let $\varepsilon \in A^{0,0}(\pt)$ be the element such that
$\Sigma_{T}^{2}\varepsilon \in Hom_{SH(S)}(T \wedge T, T \wedge T)$ is the map
exchanging the two factors $T$.  Then for $\alpha \in A^{p,q}(X,x)$ and $\beta \in A^{p',q'}(X,x)$
we have $\alpha \cup \beta = (-1)^{pp'}\varepsilon^{qq'} \beta \cup \alpha$.
In particular, $A^{4*,2*}(X,x)$ is contained in the center of $A^{*,*}(X,x)$.
%The cohomology theory $E^{\ast,\ast}$ defined by an oriented commutative
%ring $T$-spectrum is a graded commutative ring cohomology theory.

%Occasionally a ring $T$-spectrum $(E,\mu,e)$ might have a model
%$(E',\mu',e')$ which is a symmetric ring $T$-spectrum, that is, a
%symmetric $T$-spectrum $E'$ equipped with a strict multiplication $\mu'\colon E'\wedge E'
%\to E'$
%which is strictly associative and strictly unital for the unit
%$e'\colon \Sigma^\infty_{T}(S_+) \to E'$. This is the case for the algebraic
%cobordism ring $T$-spectrum $\MGL$, as described below. Such a model for the
%algebraic $K$-theory ring $T$-spectrum $\mathrm{BGL}$ is currently not known to us.

%For the rest of the paper let $k$ be a field and $S=\mathrm{Spec}(k)$.
%Usually $S$ will be replaced by $k$ in the notation.
%We will write
%$H_\bullet (k)$
%and
%$\mathrm{SH(k)}$
%for the
%$H_\bullet (S)$
%and
%$SH(S)$
%respectively.
We work in this text with the algebraic cobordism
$T$-spectrum
$\MSp$ of \cite[\S 6]{Panin:2010ab}
and the hermitian $K$-theory
$T$-spectrum
$\BO$ of \cite[\S 8]{Panin:2010aa}.
The spectrum
$\MSp$
is a commutative ring $T$-spectrum because it
is naturally a commutative monoid in the category of symmetric $T^{\wedge 2}$-spectra.
The $T$-spectrum
$\BO$
has a commutative monoid structure as shown in
\cite[Theorem 1.3]{Panin:2010aa}.

\section{The first Borel class {}{$b_{1}(E,\phi)$}{p\_1(E,phi)}}
%%%%%%%%%%%%%%%%%%%%%%%%%%%%%%%%%%%%%%%%%%%%%%%%%%%%%

Let $V$ be a vector bundle over a smooth
$S$-scheme $X$ with zero section $z\colon X \hra V$.
The \emph{Thom space} of $V$ is the quotient motivic space $\Th V = V/(V-z(X))$.
It is pointed by the image of $V-z(X)$.  It comes with a canonical
structure map $z \colon X_{+} \to \Th V$ induced by the zero section.
For the trivial bundle $\Aff^{n} \to \pt$ one has $\Th \Aff^{n} = T^{\wedge n}$.

We write $\HH$ for the trivial rank $2$ symplectic bundle
$\left( \OO^{\oplus 2}, \bigl( \begin{smallmatrix} 0 & 1 \\ -1 & 0 \end{smallmatrix} \bigr) \right)$.
The orthogonal direct sum $\HH^{\oplus n}$ is the trivial symplectic bundle of rank $2n$.

The most basic form a symplectic orientation is a symplectic Thom structure
\cite[Definition 7.1]{Panin:2010fk}.  We will use the following version of the definition.

\begin{defn}
\label{D:sp.thom}
Let $(A,\mu,e)$ be a symmetric ring $T$-spectrum.
A \emph{symplectic Thom structure} on the cohomology theory
$A^{*,*}$ is a
rule which assigns to each rank $2$ symplectic bundle $(E,\phi)$ over an $X$
in $\Sm/S$ an element
$\thom(E,\phi) \in A^{4,2}(\Th E) = A^{4,2}(E,E-X)$ with the following properties:
\begin{enumerate}
\item For an isomorphism $u \colon (E,\phi) \cong (E_{1},\phi_{1})$
one has $\thom(E,\phi) = u^{*}\thom(E_{1},\phi_{1})$.

\item For a morphism $f \colon Y \to X$ with pullback map $f_{E} \colon f^{*}E \to E$
one has $f_{E}^{*}\thom(E,\phi) = \thom(f^{*}E,f^{*}\phi)$.

\item For the rank $2$ trivial symplectic bundle $\HH$ over $\pt $ the map
\[
{-}\times \thom(\HH) \colon
A^{*,*}(X) \to A^{*+4,*+2}(X \times \Aff^{2},X \times (\Aff^{2}-0))
\]
is an isomorphism for all $X$.
\end{enumerate}
The \emph{Borel class} of $(E,\phi)$ is
$b_{1}(E,\phi) = -z^{*} \thom(E,\phi) \in A^{4,2}(X)$ where
$z \colon X \to E$ is the zero section.
\end{defn}

The sign in the Borel class is simply conventional.  It is chosen so that if $A^{*,*}$ is an oriented
cohomology theory with an additive formal group law, then the Chern and Borel classes satisfy the
traditional formula $b_{i}(E,\phi) = (-1)^{i}c_{2i}(E)$.

From Mayer-Vietoris one sees that for any rank $2$ symplectic bundle
\begin{equation*}
%\label{E:Thom.iso}
{}\cup \thom(E,\phi) \colon A^{*,*}(X) \xra{\cong} A^{*,*}(E,E-X)
\end{equation*}
is an isomorphism.

The \emph{quaternionic Grassmannian}
$HGr(r,n) = HGr(r,\HH^{\oplus n})$ is defined as the open subscheme of
$Gr(2r,2n) = Gr(2r,\HH^{\oplus n})$ parametrizing subspaces of dimension $2r$ of the fibers of
$\HH^{\oplus n}$
on which the symplectic form of $\HH^{\oplus n}$ is nondegenerate.  We write $\shf U_{HGr(r,n)}$ for
the restriction to $HGr(r,n)$ of the tautological subbundle of $Gr(2r,2n)$.  The symplectic form of
$\HH^{\oplus n}$ restricts to a symplectic form on $\shf U_{HGr(r,n)}$ which we
denote by $\phi_{HGr(r,n)}$.  The pair $(\shf U_{HGr(r,n)},\phi_{HGr(r,n)})$ is the
\emph{tautological symplectic subbundle} of rank $2r$ on $HGr(r,n)$.

More generally, given a symplectic bundle $(E,\phi)$ of rank $2n$ over $X$, the
\emph{quaternionic Grassmannian bundle} $HGr(r,E,\phi)$ is the open subscheme of
the Grassmannian bundle $Gr(2r,E)$ parametrizing subspaces of dimension $2r$
of the fibers of $E$ on which $\phi$
is nondegenerate.

For $r=1$ we have \emph{quaternionic projective spaces} and \emph{bundles}
$HP^{n} = HGr(1,n+1)$ and $HP(E,\phi) = HGr(1,E,\phi)$.

The quaternionic projective bundle theorem is proven in \cite{Panin:2010fk} using the symplectic Thom
structure and not any other version of a symplectic orientation.  It is proven first for trivial bundles.

\begin{thm}[\protect{\cite[Theorem 8.1]{Panin:2010fk}}]
\label{T:H.proj.bdl.1}
Let $(A,\mu,e)$ be a commutative ring $T$-spectrum with
a symplectic Thom structure on $A^{*,*}$.
Let $(\mathcal U_{HP^{n}}, \phi_{HP^{n}})$ be the tautological rank $2$
symplectic subbundle over $HP^n$ and
$t = b_{1}(\mathcal U_{HP^{n}}, \phi_{HP^{n}}) \in A^{4,2}(HP^{n})$ its Borel class.
Then for any $X$ in $\Sm/S$ we have an isomorphism of bigraded rings
\[
A^{*, *}(HP^n \times X) \cong A^{*,*}(X)[t]/(t^{n+1}).
\]
\end{thm}

A Mayer-Vietoris argument gives the more general theorem \cite[Theorem 8.2]{Panin:2010fk}.

\begin{thm}[\protect{Quaternionic projective bundle theorem}]
\label{T:H.proj.bdl.2}
Let $(A,\mu,e)$ be a commutative ring $T$-spectrum with
a symplectic Thom structure on $A^{*,*}$.
Let $(E,\phi)$  be a symplectic bundle of rank $2n$ over $X$, let
$(\mathcal U, \phi|_{\mathcal U})$ be the tautological rank $2$
symplectic subbundle over
the quaternionic projective bundle $HP(E, \phi)$, and let
$t = b_{1}(\mathcal U, \phi|_{\mathcal U})$
be its Borel class. Then
we have an isomorphism of bigraded $A^{*,*}(X)$-modules
\[
(1, t, \dots , t^{n-1}) \colon
 A^{*,*}(X) \oplus A^{*,*}(X) \oplus \dots \oplus
A^{*,*}(X) \to A^{*,*}(HP(E, \phi)).
\]
\end{thm}

\begin{defn}
\label{D:BorelClasses}
Under the hypotheses of Theorem
\ref{T:H.proj.bdl.2}
there are unique elements
$b_i(E, \phi) \in A^{4i,2i}(X)$ for $i=1,2, \dots , n$
such that
$$t^{n}-b_1(E, \phi)\cup t^{{n-1}} + b_2(E, \phi)\cup t^{{n-2}} - \dots + (-1)^n b_n(E, \phi)=0.$$
The classes
$b_i(E, \phi)$
are called the \emph{Borel classes}
of $(E, \phi)$ with respect to the symplectic Thom structure of the cohomology theory $(A, \partial)$.
For $i > n$ one sets $b_i(E, \phi)$ = 0, and one sets $b_0(E, \phi) = 1$.
\end{defn}

%For a rank $2$ symplectic bundle $(E,\phi)$ the classes $b_{1}(E,\phi)$ defined by Definitions
%\ref{D:sp.thom}   and \ref{D:BorelClasses} coincide.

\begin{cor}
\label{C:pont.triv}
The Borel classes of a trivial symplectic bundle vanish\textup{:} $b_{i}(\HH^{\oplus n}) = 0$.
\end{cor}
%
%
%Among the consequences of the quaternionic projective bundle theorem is the symplectic splitting
%principle \cite[Theorem 10.2]{Panin:2010fk}.

The Cartan sum formula holds for Borel classes \cite[Theorem 10.5]{Panin:2010fk}.
In particular:

\begin{thm}
%[\protect{Cartan sum formula }]
\label{T:p1.sum}
Let $(A,\mu,e)$ be a commutative ring $T$-spectrum with
a symplectic Thom structure on $A^{*,*}$.
Let $(E,\phi)$ and $(F,\psi)$ be symplectic bundles over $X$.  Then we have
\begin{equation}
\label{E:sum.formula}
b_{1}\bigl( (E,\phi) \oplus (F,\psi) \bigr) = b_{1}(E,\phi) + b_{1}(F,\psi).
\end{equation}
\end{thm}

%The Borel classes are compatible with base change.  This implies that they are $\Aff^{1}$-deformation invariant in the following sense.
%
%\begin{prop}
%\label{P:deform.invariant}
%Let $(E_{0},\phi_{0})$ and $(E_{1},\phi_{1})$ be symplectic bundles over $X$.
%Suppose there exists a symplectic bundle $(E,\phi)$ on
%$X \times \Aff^{1}$ with $(E, \phi) |_{S \times \{0\} } \cong (E_{0},\phi_{0})$ and
% $(E, \phi) |_{X \times \{1\} } \cong (E_{1},\phi_{1})$.
% Then we have $b_{i}(E_{0},\phi_{0}) = b_{i}(E_{1},\phi_{1})$ for all $i$.
%\end{prop}

%A vector bundle $ L$ has an associated \emph{hyperbolic symplectic bundle}
%$\HH( L) = \left(  L \oplus  L^{\vee} ,
%\bigl( \begin{smallmatrix} 0&1\\-1&0 \end{smallmatrix} \bigr) \right)$.

We also have the following result \cite[Proposition 8.5]{Panin:2010fk}.

\begin{prop}
\label{P:subl}
Suppose that $( E, \phi)$ is a symplectic bundle over
$X$ with a totally isotropic subbundle $L \subset  E$.
%Let
%$( E_0, \phi_0) = ( L^\perp/ L, \overline \phi) \oplus
%\HH( L)$.
Then for all $i$ we have
\[
b_i( E,\phi) = b_{i} \left( ( L^\perp/ L, \overline \phi) \oplus
\bigl(  L \oplus  L^{\vee} ,
\bigl( \begin{smallmatrix} 0&1_{L^{\vee}}\\-1_{L}&0 \end{smallmatrix} \bigr) \bigr) \right).
\]
\end{prop}

This is because there is an $\Aff^{1}$-deformation between the two symplectic bundles.
%and that that is enough.

%This is because there exists a symplectic bundle $( E_1, \phi_1)$
%over $X \times \Aff^1$ with $( E_1, \phi_1) |_ {X \times \{t\}
%} \cong ( E, \phi)$ for $t \neq 0$ and $( E_1, \phi_1) |_{
%  X \times \{ 0 \} } \cong ( E_0, \phi_0)$.

\begin{defn}
\label{D:GW.bundles}
The \emph{Grothendieck-Witt group of symplectic bundles}  $GW^{-}(X)$
is the abelian group of formal differences $[E,\phi]-[F,\psi]$ of symplectic vector bundles
over $X$ modulo three relations:
\begin{enumerate}
\item For an isomorphism $u \colon (E,\phi) \cong (E_{1},\phi_{1})$
one has $[E,\phi] = [E_{1},\phi_{1}]$.
\item For an orthogonal direct sum one has $[(E,\phi) \oplus (E_{1},\phi_{1})]
= [E,\phi] + [E_{1},\phi_{1}]$.
\item If $( E, \phi)$ is a symplectic bundle over
$X$ with a totally isotropic subbundle $L \subset  E$, then we have
$[E,\phi] = [L^\perp/ L, \overline \phi] +
\left[  L \oplus  L^{\vee} ,
\bigl( \begin{smallmatrix} 0&1\\-1&0 \end{smallmatrix} \bigr) \right]$.
\end{enumerate}

The \emph{Grothendieck-Witt group of orthogonal bundles} $GW^{+}(X)$ is defined analogously.
\end{defn}

\begin{thm}
Let $(A,\mu,e)$ be a commutative ring $T$-spectrum with
a symplectic Thom structure on $A^{*,*}$.  Then the associated first Borel class
induces a well-defined additive map
\[
b_{1} \colon GW^{-}(X) \to A^{4,2}(X)
\]
which is functorial in $X$.
\end{thm}

%In the derived category $D^{b}(VB_{X})$ of bounded chain complexes of vector over $X$ there is
%the standard duality $({}^{\vee},\eta)$ composed of the duality functor $\shf E \mapsto \shf E^{\vee}$
%and the standard biduality isomorphisms $\eta_{\shf E} \colon \shf E \cong \shf E^{\vee\vee}$.
%
%In \cite{Balmer:1999if} Balmer introduced shifted dualities for chain complexes over schemes,
%and defined \emph{symmetric complexes} $(E,\phi)$ with respect to these shifted dualities.  He also
%defined a notion of a lagrangian for such a complex.

In \cite{Schlichting:2010vn} Schlichting constructed hermitian $K$-theory spaces for
exact categories.  This gives hermitian $K$-theory spaces $KO(X)$ and $KSp(X)$ for
orthogonal and symplectic bundles on schemes.  Their $\pi_{0}$ are $GW^{+}(X)$ and $GW^{-}(X)$
respectively.
In \cite{Schlichting:2010uq} he constructed Hermitian $K$-theory spaces $KO^{[m]}(X,U)$
for complexes of vector bundles on $X$ acyclic on the open subscheme $U$ equipped with a
nondegenerate symmetric bilinear form for the duality shifted by $m$.  For an even integer
$2n$ an orthogonal bundle $(U,\psi)$ gives a chain complex $U[2n]$ equipped with a nondegenerate
symmetric bilinear form $\psi[4n] \colon U[2n] \otimes_{\OO_{X}} U[2n] \to \OO_{X}[4n]$
in the symmetric monoidal category $D^{b}(VB_{X})$.
For an odd integer $2n+1$ a symplectic bundle $(E,\phi)$ gives a chain complex $E[2n{+}1]$
equipped with a nondegenerate symmetric bilinear form
$\phi[4n{+}2]\colon E[2n{+}1]\otimes_{\OO_{X}} E[2n{+}1] \to \OO_{X}[4n{+}2]$.  These functors
induce homotopy equivalences of spaces
$KO(X) \to KO^{[4n]}(X)$ and $KSp(X) \to KO^{[4n{+}2]}(X)$
\cite[Proposition 6]{Schlichting:2010uq}.

The simplicial presheaves $X \mapsto KO^{[n]}(X)$ are pointed motivic spaces.  D\'evissage gives
schemewise weak equivalences $KO^{[n]}(X) \to KO^{[n+1]}(X \times \Aff^{1},X \times (\Aff^{1}-0))$
which are adjoint to maps $KO^{[n]} \times T \to KO^{[n+1]}$.   These are the structural maps of a
$T$-spectrum $(KO^{[0]},KO^{[1]},KO^{[2]},\dots)$ of which our $\BO$ is a fibrant replacement
\cite[\S\S7--8]{Panin:2010aa}.  One has $KO_{i}^{[n]}(X,U) = \BO^{4n-i,2n}(X_{+}/U_{+})$
for all $i \geq 0$ and $n$.  Hence $\BO^{4n,2n}(X_{+}/U_{+})$ is the Grothendieck-Witt group
for the usual duality shifted by $n$ of symmetric chain complexes of vector bundles on $X$ which
are acyclic on $U$.

\begin{defn}
\label{D:proper.iso}
The \emph{right isomorphisms} are
\begin{eqnarray*}
\unsigntrans _{4n} \colon \;\; GW^{+}(X) & \overset{\cong}{\lra} &
KO_{0}^{[4n]}(X) =
\BO^{8n,4n}(X)
\\ {}
[U,\psi] & \longmapsto &
\bigl[ U[2n],\psi[4n] \bigr]
\end{eqnarray*}
and
\begin{eqnarray*}
\signtrans _{4n+2} \colon \;\; GW^{-}(X) & \overset{\cong}{\lra} &
KO_{0}^{[4n+2]}(X) =
\BO^{8n+4,4n+2}(X)
\\ {}
[E,\phi] & \longmapsto
& -\bigl[ E[2n+1], \phi[4n+2] \bigr]
\end{eqnarray*}
\end{defn}

The sign in $\signtrans_{4n+2}$ is chosen so that it commutes with the forgetful maps
to  $K_{0}(X)$, where we have $[E] = -\bigl[ E[2n+1] \bigr]$.
Most authors of papers on Witt groups do not use this sign because Witt groups do not
have forgetful maps to $K_{0}(X)$.

\begin{defn}
\label{D:periodicity}
The \emph{periodicity elements} $\beta_{8} \in \BO^{8,4}(\pt)$ and $\beta_{8}^{-1} \in \BO^{-8,-4}(\pt)$
%is the element which
correspond to the unit $1 = [\OO_{X},1] \in GW^{+}(X)$ under the
%right
isomorphisms $\BO^{8,4}(\pt) \cong GW^{+}(\pt) \cong \BO^{-8,-4}(\pt)$
%$\unsigntrans_{4}$ and $\unsigntrans_{-4}$
of Definition \ref{D:proper.iso}.
\end{defn}

We have the composition
\begin{equation}
\label{E:p1.A}
\widetilde b_{1}^{A} \colon \BO^{4,2}(X)
%\underset{\cong}{\overset{t_{2}}{\longleftarrow}}
\xla[\cong]{\signtrans_{2}}
GW^{-}(X)
\overset{b_{1}}{\lra}
%\xra{b_{1}}
A^{4,2}(X)
\end{equation}

The Thom classes for hermitian $K$-theory are constructed by the
same method that Nenashev used for Witt groups \cite[\S 2]{Nenashev:2007rm}.
Suppose we have an $SL_{n}$-bundle
 $(E,\lambda)$ consisting of a vector bundle $\pi \colon E \to X$ of rank $n$
and $\lambda \colon \OO_{X} \cong \det E$ an isomorphism of line bundles.
The pullback $\pi^{*}E = E \oplus E \to E$ has a canonical section $\Delta_{E}$, the diagonal.
There is a Koszul complex
\[
K(E) = \bigl( 0 \to \Lambda^{n} \pi^{*}E^{\vee} \to \Lambda^{n-1} \pi^{*}E^{\vee} \to \cdots \to
\Lambda^{2} \pi^{*}E^{\vee} \to E^{\vee} \to \OO_{E} \to 0 \bigr)
\]
%(considered as a chain complex in homological degrees $n$ to $0$)
in which each boundary map the contraction with $\Delta_{E}$.    It is a locally free
resolution of the coherent sheaf $z_{*}\OO_{X}$ on $E$.  There is a canonical isomorphism
$\varTheta(E,\lambda) \colon K(E) \to K(E)^{\vee} [n]$ induced by $\lambda$
which is symmetric for the shifted duality.

\begin{defn}
\label{D:BO.thom}
In the \emph{standard special linear Thom structure} on $\BO$ the Thom
class of the special linear bundle $(E,\lambda)$ of rank $n$ is
\begin{equation*}
\thom^{\BO}(E,\lambda) = [K(E),\varTheta(E,\lambda)] \in KO_{0}^{[n]}(E,E-X) = \BO^{2n,n}(E,E-X)
\end{equation*}
In the \emph{standard symplectic Thom structure} on $\BO$ the Thom
class of the symplectic bundle $(E,\phi)$ of rank $2r$ is
\begin{equation*}
\thom^{\BO}(E,\phi) = \thom^{\BO}(E, \lambda_{\phi}) \in \BO^{4r,2r}(E,E-X)
\end{equation*}
for $\lambda_{\phi} = (\Pf \phi)^{-1}$ where $\Pf \phi \in \Gamma(X,\det E^{\vee})$ denotes
the Pfaffian of $\phi \in \Gamma(X,\Lambda^{2} E^{\vee})$.
\end{defn}

The corresponding first Borel class of a rank $2$ symplectic bundle is therefore
\[
%\begin{multline*}
%\label{E:p1.BO.1}
b_{1}^{\BO}(E,\phi) = -[K(E) ,\varTheta(E,\lambda_{\phi})] |_{X}
% =
%-\bigl[ E^{\vee}[1], \phi^{-1}[2] \big] - \left[ \OO_{X} \oplus \OO_{X}[2], \left(
%\begin{smallmatrix}
%0 & 1[2] \\ -1 & 0
%\end{smallmatrix}
%\right) \right]
%\\
%=
%-\bigl[ E^{\vee}[1], \phi^{-1}[2] \big] + \left[ \OO_{X}[1] \oplus \OO_{X}[1], \left(
%\begin{smallmatrix}
%0 & 1[1] \\ -1[1] & 0
%\end{smallmatrix}
%\right) \right]
\in \BO^{4,2}(X).
%\end{multline*}
\]
%
%We may regard $K_{0}(X)$ as the Grothendieck group of bounded chain complexes of vector bundles
%on $X$.  Then there is a \emph{hyperbolic map}
%%
%\begin{eqnarray*}
%h_{2}\colon \;\; K_{0}(X) & \lra & KO_{0}^{[2]}(X)
%\\{}
%[\shf F] \mspace{12mu} & \longmapsto & \left[ \shf F \oplus \shf F^{\vee}[2] ,
%\left(
%\begin{smallmatrix}
%0 & 1_{\shf F^{\vee}[2]} \\ -1_{\shf F} & 0
%\end{smallmatrix}
%\right)\right]
%\end{eqnarray*}
%
A short calculation shows that this is the class which corresponds to
$[E,\phi] - [\HH] \in GW^{-}(X)$ under the
%right
isomorphism $\signtrans_{2}$.
%of Definition \ref{D:proper.iso}.
The symplectic splitting principle \cite[Theorem 10.2]{Panin:2010fk} and Theorem \ref {T:p1.sum} now give
the next proposition.

\begin{prop}
\label{P:p1.BO}
Let $(E,\phi)$ be a symplectic bundle of rank $2r$ on $X$.  Then $b_{1}^{\BO}(E,\phi) \in \BO^{4,2}(X)$
is the class which corresponds to $[E,\phi]-r[\HH] \in GW^{-}(X)$ under the
%right
isomorphism $\signtrans_{2}$.
% of Definition \ref{D:proper.iso}.
%
%Let $X = \bigsqcup X_{i}$ be the connected components of $X$, and let $\HH_{i}$ be the
%symplectic bundle which is trivial of rank $2$ on $X_{i}$ and which vanishes on $X - X_{i}$.
%Then $b_{1}^{\BO}(E,\phi) \in \BO^{4,2}(X)$ is the class which corresponds under the right
%isomorphism of Definition \ref{D:proper.iso} to
%$[E,\phi] - \sum_{i} \frac 12 (\rk_{X_{i}} E ) [\HH_{i}] \in GW^{-}(X)$.
\end{prop}

Let $X = \bigsqcup X_{i}$ be the connected components of $X$.  We consider the elements
and functions
\begin{align}
\label{E:h}
1_{X_{i}} \in \BO^{0,0}(X),
&&
\rk_{X_{i}} \colon \BO^{4,2}(X) \to \ZZ,
&&
\hh \in \BO^{4,2}(\pt).
\end{align}
The first is the central idempotent which is the image of the unit
$1_{X_{i}} \in \BO^{0,0}(X_{i})$.
The second is the rank function on the Grothendieck-Witt group $KO_{0}^{[2]}(X)$
of bounded chain complexes of vector bundles.  The third is the class corresponding
to $[\HH] \in GW^{-}(\pt)$ under the
right
isomorphism $\signtrans_{2} \colon GW^{-}(\pt) \cong \BO^{4,2}(\pt)$.

Let $\widetilde b_{1}^{\BO} \colon \BO^{4,2}(X) \to \BO^{4,2}(X)$ be the map of
\eqref{E:p1.A}.

\begin{cor}
\label{C:p1}
For all $\alpha \in \BO^{4,2}(X)$ we have
$\alpha = \widetilde b_{1}^{\BO}(\alpha) + \hh \,\prod_{i} \frac 12 (\rk_{X_{i}}\alpha )1_{X_{i}}$.
\end{cor}

\section{Symplectically oriented commutative ring {}{$T$}{T}-spectra}
%%%%%%%%%%%%%%%%%%%%%%%%%%%%%%%%%%%%%%%%%%%%%%%%%%%%%

%Following Adams and Morel we define an orientation of a commutative
%ring $T$-spectrum. However we prefer to use Thom classes instead of
%Chern classes.

Embed $\HH^{\oplus n} \subset\HH^{\oplus \infty}$ as the direct sum
of the first $n$ summands.
%Let $\shift \colon \HH^{\oplus \infty} \to \HH^{\oplus \infty}$
%be endomorphism which sends the $n^{\text{th}}$ summand onto the $(n+1)^{\text{st}}$.
The ensuing filtration $\HH \subset \HH^{\oplus 2} \subset \HH^{\oplus 3} \subset \cdots$
for each $r$ a direct system of schemes
\[
\pt = HGr(r,r) \hra HGr(r,r+1) \hra HGr(r,r+2) \hra \cdots.
\]
The ind-scheme and motivic space
\[
BSp_{2r} = HGr(r,\infty) = \colim_{n \geq r} HGr(r,n)
\]
is pointed
by $h_{r} \colon \pt = HGr(r,r) \into BSp_{2r}$.   Each $HGr(r,n)$ has a tautological
symplectic subbundle $(\shf U_{HGr(r,n)},\phi_{HGr(r,n)})$, and their colimit is an ind-scheme
$\shf U_{BSp_{2r}}$ which is a vector bundle over the ind-scheme $BSp_{2r}$.
It has a Thom space $\Th \shf U_{BSp_{2r}}$ just like for ordinary schemes.  We write
\begin{equation*}
\label{E:MSp}
\MSp_{2r} = \Th \shf U_{BSp_{2r}} = \Th \shf U_{HGr(r,\infty)} =
\colim_{n \geq r} \Th \shf U_{HGr(r,n)}.
\end{equation*}
We refer the reader to \cite[\S 6]{Panin:2010ab} for the complete construction of $\MSp$
as a commutative monoid in the category of symmetric $T^{\wedge 2}$-spectra.  The unit comes
from the pointings $h_{r}\colon \pt \into BSp_{2r}$, which induce canonical inclusions of
Thom spaces
\[
e_{r} \colon
T^{\wedge 2r} \into \MSp_{2r}.
\]

Let $(A,\mu,e)$ be a commutative ring $T$-spectrum. The unit
%$e \in Hom_{SH(S)}(\pt_{+},A)$
of the monoid defines the unit element
$1_{A}\in A^{0,0}(\pt_+)$.
% of the product on cohomology.
Applying the
$T$-suspension
isomorphism twice gives an element
$\Sigma^2_{T}1_{A} \in A^{4,2}(T^{\wedge 2}) = A^{4,2}(\Th \Aff^{2})$.

\begin{defn}
A \emph{symplectic Thom orientation} on a commutative ring $T$-spectrum $(A,\mu,e)$ is
an element $\thom \in A^{4,2}(\MSp_{2}) = A^{4,2}(\Th \shf U_{HP^{\infty}})$ with
$\thom |_{T^{\wedge 2}} =  \Sigma_{T}^{2}1_{A} \in A^{4,2}(T^{\wedge 2})$.
\end{defn}

The element $\thom$ should be regarded as the symplectic Thom class of the tautological
quaternionic line bundle
$\shf U_{HP^{\infty}}$
%$\mathcal{T}(2)$
over $HP^{\infty}$.
%The element $c$ should be regarded as a Chern class of the tautological
%line bundle
%$\mathcal{T}(1)= \mathcal O(-1)$
%over
%$\Pro^{\infty}$.
%\end{rem}

\begin{example}
\label{Ex:thom.MSp}
The \emph{standard symplectic Thom orientation} on algebraic symplectic cobordism is the element
$\thom^{\MSp} = u_{2} \in \MSp^{4,2}(\MSp_{2})$ corresponding to the canonical map
$u_{2} \colon \Sigma_{T}^{\infty}\MSp_{2}(-2) \to \MSp$ described in \eqref{E:canonical.map}.
\end{example}

The main theorem of \cite{Panin:2010ab} gives seven other structures containing the same information
as a symplectic Thom orientation.  First:

\begin{thm}
[\protect{\cite[Theorem 10.2]{Panin:2010ab}}]
\label{T:a.alpha}
Let $(A,\mu,e)$ be a commutative monoid in $SH(S)$.  There is a canonical bijection between the sets
of

\parens{a} symplectic Thom structures on the ring cohomology theory
$A^{*,*}$ such that for the trivial rank $2$ symplectic bundle $\HH$ over $\pt$ we have
$\thom(\HH) = \Sigma_{T}^{2}1_{A}$ in $A^{4,2}(T^{\wedge 2})$, and

%\parens{b} Borel structures on
%%the bigraded $\epsilon$-commutative ring cohomology theory
%$(A^{*,*},\partial,\times,1_{A})$ for which
%$b_{1}(\shf U_{HP^{1}},\phi_{HP^{1}}) \in A^{4,2}(HP^{1},h_{\infty}) \subset A^{4,2}(HP^{1})$
%corresponds to $-\Sigma_{T}^{2}1_{A}$ in $A^{4,2}(T^{\wedge 2})$ under the canonical
%motivic homotopy equivalence $(HP^{1},h_{\infty}) \simeq T^{\wedge 2}$,
%
%\parens{c} Borel classes theories on $(A^{*,*},\partial,\times, 1_{A})$
%with the same normalization condition on $b_{1}(\shf U_{HP^{1}},\phi_{HP^{1}})$ as in \parens{b},
%
%\parens{d} symplectic Thom classes theories on
%%the bigraded $\epsilon$-commutative ring cohomology theory
%$(A^{*,*},\partial,\times,1_{A})$ such that for the trivial rank $2$
%bundle $\Aff^{2} \to \pt$ we have
%$\thom(\Aff^{2},\omega_{2}) = \Sigma_{T}^{2}1_{A}$ in $A^{4,2}(T^{\wedge 2})$,

\parens{$\alpha$}
symplectic Thom orientations on $(A,\mu,e)$.
%, i.e.~
%classes $\vartheta \in A^{4,2}(\MSp_{2})$ with
%$\vartheta |_{T^{\wedge 2}} = \Sigma_{T}^{2}1_{A}$ in $A^{4,2}(T^{\wedge 2})$.

%\parens{$\beta$} classes $\varrho \in A^{4,2}(HP^{\infty},h_{\infty})$ with
%$\varrho|_{HP^{1}} \in A^{4,2}(HP^{1},h_{\infty})$ corresponding to
%$-\Sigma_{T}^{2}1_{A} \in A^{4,2}(T^{\wedge 2})$ under the canonical motivic homotopy equivalence
%$(HP^{1},h_{\infty}) \cong T^{\wedge 2}$,
%
%\parens{$\delta$} sequences of classes
%$\boldsymbol{\vartheta} = (\vartheta_{1},\vartheta_{2},\vartheta_{3},\dots)$
%with $\vartheta_{r} \in A^{4r,2r}(\MSp_{2r})$ for each $r$ satisfying
%$\mu_{rs}^{*} \vartheta_{r+s} = \vartheta_{r} \times \vartheta_{s}$ for all $r,s$,
%and $\vartheta_{1} |_{T^{\wedge 2}} = \Sigma_{T}^{2}1_{A}$,

%\parens{$\varepsilon$} morphisms $\varphi \colon (\MSp,\mu_{\MSp},e_{\MSp}) \to (A,\mu,e)$
%of commutative monoids
%in $SH(S)$.
\end{thm}

Thus a symplectic Thom orientation determines Thom and Borel classes for all symplectic bundles.

%For the details see \cite[Theorems 10.2, 12.3, 13.2]{Panin:2010ab}.

\begin{lem}
In the standard special linear and symplectic Thom structures on $\BO$ we have
$\thom(\Aff^{1},1) = \Sigma_{T}1_{\BO}$ and $\thom(\HH) =
%\thom(\Aff^{1},1)^{\cup 2} =
\Sigma_{T}^{2}1_{\BO}$.
\end{lem}

\begin{proof}
The structural maps
$\KO^{[n]}\wedge T \to \KO^{[n+1]}$ of the spectrum are by definition \cite[\S 8]{Panin:2010aa}
adjoint to maps
$\KO^{[n]} \to \operatorname{\mathbf{Hom}}_{\bullet}(T,\KO^{[n+1]})$ which are
fibrant replacements of maps of simplicial presheaves
\[
({-}\boxtimes (K(\OO),\varTheta(\OO,1)))_{*} \colon
KO^{[n]}({-}) \to KO^{[n+1]}({-} \wedge T)
\]
which act on the homotopy groups as ${-}\cup [K(\OO),\varTheta(\OO,1)] = {-} \cup \thom(\Aff^{1},1)$.
So we have $\Sigma_{T}1_{\BO} = \thom(\Aff^{1},1)$.
It then follows that we have $\thom(\HH) = \thom(\Aff^{1},1)^{\cup 2} = \Sigma_{T}^{2}1_{\BO}$.
\end{proof}

The standard symplectic Thom structure on $\BO$ thus satisfies the normalization condition of
Theorem \ref{T:a.alpha}.  It corresponds to the
%
%\begin{example}
%\label{Ex:thom.BO}
%The
\emph{standard symplectic Thom orientation} on hermitian $K$-theory
$\thom^{\BO} \in \BO^{4,2}(\MSp_{2})$.
It is given by the formulas of Definition \ref{D:BO.thom} for
$(E,\phi) = (\shf U_{HP^{\infty}},\phi_{HP^{\infty}})$ tautological subbundle on $HP^{\infty} = BSp_{2}$.

A \emph{symplectically oriented commutative $T$-ring spectrum} is a pair $(A,\vartheta)$
with $A$ a commutative monoid in $SH(S)$ and $\vartheta$ a symplectic
Thom orientation on $A$.  We could write the associated Thom and Borel classes as
$\thom^{\vartheta}(E,\phi)$ and $b_{i}^{\vartheta}(E,\phi)$.

A \emph{morphism of symplectically oriented commutative $T$-ring spectra}
$\varphi \colon (A,\vartheta) \to (B,\varpi)$ is a morphism of commutative monoids with
$\varphi(\vartheta) = \varpi$.
%Such a $\varphi$ sends the Thom and Borel classes of $(A,\vartheta)$ to the Thom and
%Borel classes of $(B,\varpi)$.
For such a $\varphi$ one has $\varphi(\thom^{\vartheta}(E,\phi)) = \thom^{\varpi}(E,\phi)$
and $\varphi(b_{i}^{\vartheta}(E,\phi)) = b_{i}^{\varpi}(E,\phi)$ for all symplectic bundles.

\begin{thm}
[Universality of $\MSp$]
\label{T:main.Sp}
Let $(A,\mu,e)$ be a commutative monoid in $SH(S)$.
The assignments $\varphi \mapsto \varphi(\thom^{\MSp})$ gives a bijection between the sets of

\parens{$\varepsilon$} morphisms $\varphi \colon (\MSp,\mu_{\MSp},e_{\MSp}) \to (A,\mu,e)$
of commutative monoids in $SH(S)$, and

\parens{$\alpha$}
symplectic Thom orientations on $(A,\mu,e)$.
\end{thm}

This is \cite[Theorems 12.3, 13.2]{Panin:2010ab}.
Thus $(\MSp,\thom_{\MSp})$ is the universal symplectically oriented commutative $T$-ring spectrum.

Let $\varphi \colon (A,\vartheta) \to (B,\varpi)$ be
a morphism of symplectically oriented commutative $T$-ring spectra.
For a space $X$ the isomorphisms $X \wedge \pt_{+} \cong X \cong \pt_{+} \wedge X$
make $A^{*,*}(X)$ into a two-sided module over the ring $A^{*,*}(\pt)$ and into a
bigraded-commutative algebra over the commutative ring $A^{4*,2*}(\pt)$.  The morphism $\varphi$
induces morphisms of graded rings
\begin{equation}
\label{E:CF.hom}
\begin{gathered}
%\label{CFhomPeriodic.1}
\bar \varphi
%^{\ast,\ast}
_{X}\colon A^{\ast,\ast}(X) \otimes_{A^{4*,2*}(\pt)}
B^{4*,2*}(\pt) \to
B^{\ast,\ast}(X)
\\
%\label{CFhomPeriodic.2}
\bar \varphi
%^{4*,2*}
_{X}\colon A^{4*,2*}(X) \otimes_{A^{4*,2*}(\pt)}
B^{4*,2*}(\pt) \to
B^{4*,2*}(X)
\end{gathered}
\end{equation}
which are natural in $X$, with the pullbacks acting on the left side of the $\otimes$.

\begin{thm}
[Weak quaternionic cellularity of $\MSp_{2r}$]
\label{T:cellularity}
Let $\varphi \colon (A,\vartheta) \to (B,\varpi)$ be a morphism of symplectically oriented
commutative $T$-ring spectra.  Then for all $r$ the natural morphism of graded rings
\[
\bar \varphi
%^{4*,2*}
_{\MSp_{2r}} \colon
A^{4*,2*}(\MSp_{2r}) \otimes _{A^{4*,2*}(\pt)} B^{4*,2*}(\pt) \to B^{4*,2*}(\MSp_{2r})
\]
is an isomorphism.
\end{thm}

\begin{proof}
%Let $z \colon BSp_{2r} \to \MSp_{2r}$ be the structural map of the Thom space induced by the
%zero section of the vector bundle $\shf U_{BSp_{2r}}$.
Let $t_{1},\dots,t_{r}$ be independent
indeterminates with $t_{i}$ of bidegree $(4i,2i)$.
By \cite[Theorems 9.1, 9.2, 9.3]{Panin:2010ab} there is a commutative diagram of isomorphisms
\[
\xymatrix @M=5pt @C=100pt {
A^{*,*}(\pt)[[t_{1},\dots,t_{r}]]^{\homog}
\ar[r]^-{t_{i} \mapsto b_{i}^{\vartheta}(\shf U_{BSp_{2r}},\phi_{BSp_{2r}})} _-{\cong}
\ar[d]_-{\times t_{r}} ^-{\cong}
&
A^{*,*}(BSp_{2r}) \ar[d] ^-{\cup \thom^{\vartheta} (\shf U_{BSp_{2r}},\phi_{BSp_{2r}})} _-{\cong}
\\
t_{r}A^{*,*}(\pt)[[t_{1},\dots,t_{r}]]^{\homog}
\ar[r]
%^-{t_{i} \mapsto b_{i}^{\vartheta}(\shf U_{BSp_{2r}},\phi_{BSp_{2r}})}
_-{\cong}
&
A^{*+4r,*+2r}(\MSp_{2r})
}
\]
The notation on the left refers to homogeneous formal power series.  There is a
similar diagram for $(B, \varpi)$.  The maps $\varphi \colon A^{*,*} \to B^{*,*}$
commute with the maps of the two diagrams because $\varphi$ sends the
Thom and Borel classes of $(A,\vartheta)$ onto the
Thom and Borel classes of $(B, \varpi)$.
The morphism $\bar \varphi^{4*,2*}_{\MSp_{2r}}$ is an isomorphism because
\[
t_{r}A^{4*,2*}(\pt)[[t_{1},\dots,t_{r}]]^{\homog}
\otimes _{A^{4*,2*}(\pt)} B^{4*,2*}(\pt) \to
t_{r}B^{4*,2*}(\pt)[[t_{1},\dots,t_{r}]]^{\homog}
\]
is an isomorphism.
\end{proof}

%%%%%%%%%%%%%%%%%%%%%%%%%%%%%%%%%%%%%%%%%%%%%%%%%%%%%
%%%%%%%%%%%%%%%%%%%%%%%%%%%%%%%%%%%%%%%%%%%%%%%%%%%%%

\section{Where the class {}{$b_{1}$}{p\_1} takes the place of honour}
\label{AgeneralResult}
We suppose that
$(\thy{U},\thy{u}) \to (\BO,\thom^{\BO})$
is a morphism of symplectically oriented commutative ring
$T$-spectra.
We set
\begin{gather*}
\bar {\thy{U}}^{\ast,\ast}(X)=
\thy{U}^{\ast,\ast}(X) \otimes_{\thy{U}^{4*,2*}(\pt)} \BO^{4*,2*}(\pt),
\\
\bar {\thy{U}}^{4*,2*}(X)=
\thy{U}^{4*,2*}(X) \otimes_{\thy{U}^{4*,2*}(\pt)} \BO^{4*,2*}(\pt),
\end{gather*}
and we write
$\bar \varphi_X$
for the morphisms of \eqref {E:CF.hom}.
%for the above homomorphism.
%
%\begin{defn}[Weakly $\MSp$-Cellular]
%\label{MGLWeaklyCellular}
%Let
%$(\thy{U},\thy{u})$
%be a symplectically oriented
%commutative ring
%$T$-spectrum over  $S$
%satisfying the universality property
%\ref{UniversalityDefn}.
%The symplectically universal commutative ring
%$T$-spectrum
%$(\thy{U}, \thy{u})$
%is called weakly $\MSp$-cellular if there exists an integer $N$
%such that
%the map
%$\bar \varphi^{4*,2*}_{U_{2n},\ast}$
%is an isomorphism for
%$n \geq N$ .
%\end{defn}
%
%\begin{rem}
%By the Universality Theorem
%\cite[Thm. 6.1]{Panin:2010ab}
%the
%$T$-spectrum
%$\MSp$
%is symplectically universal. That is why we choose to write
%$\MSp$-cellular in the definition above.
%The following theorem motivates two last Definitions.
%\end{rem}

\begin{thm}
\label{AlmostlyMain}
Let
$(\thy{U},\thy{u}) \to (\BO,\thom^{\BO})$
be a morphism of symplectically oriented commutative ring
$T$-spectra.
Suppose there exists an $N$ such that for all $n \geq N$ the maps
$\bar\varphi_{\thy{U}_{2n}} \colon \bar{\thy{U}}^{4i,2i}(\thy{U}_{2n}) \to
\BO^ {4i,2i} (\thy{U}_{2n})$
are isomorphisms for all $i$.  Then
for all small pointed motivic spaces $X$ and all $(p,q)$
the homomorphism
$\bar \varphi_{X} \colon \bar{\thy{U}}^{p,q}(X) \to \BO^{p,q}(X)$
is an isomorphism.
\end{thm}

%\begin{rem}
%It will be checked in the next Subsection that the pair
%$(\thy{MGL},\thy{th}^{MGL})$
%is an oriented commutative ring
%$T$-spectrum over $k$
%satisfying the Quillen universality property
%and that pair is weakly
%$MGL$-cellular as well.
%So the last Theorem shows the relation of the algebraic cobordisms
%of Voevodsky to Quillen's $K$-theory.
%\end{rem}

Before turning to the theorem itself, we prove a series of lemmas.  The first three demonstrate
the significance of the first Borel class for this problem.

\begin{lem}
\label{L:1}
The functorial map $\bar\varphi_{X} \colon \bar{\thy{U}}^{4,2}(X) \to \BO^{4,2}(X)$ has a section
$s_{X}$ which is functorial in $X$.
\end{lem}

\begin{proof}

Write $HGr = \colim_{r} HGr(r,\infty)$.
According to Theorem \cite[Theorem 10.1, (11.1)]{Panin:2010aa}
there is an isomorphism \textit{\`a la} Morel-Voevodsky
$\bar\tau \colon (\ZZ \times HGr, (0,x_{0})) \cong \KSp$ in $H_{\bullet}(S)$
such that the restrictions are
\begin{equation*}
%\label{E:tau.formula}
\bar\tau |_{\{i\} \times HGr(n,2n)}
= [\mathcal U_{HGr(n,2n)},\phi_{HGr(n,2n)}]+(i-n)[\HH]
\end{equation*}
in $KSp_{0}(HGr(n,2n)) = GW^{-}(HGr(n,2n))$.
Composing with the isomorphisms
in $H_{\bullet}(S)$
\[
(\ZZ \times HGr, (0,x_{0})) \xra{\bar\tau} \KSp \xrightarrow{\text{\textsf{trans}}_{1}}
\KO^{[2]}
\xrightarrow{-1} \KO^{[2]}.
\]
where the $\text{\textsf{trans}}_{1}$ comes from the translation functor
$(\mathcal F, \phi) \mapsto (\mathcal F[1], \phi[2])$, and the $-1$
is the inverse operation of the $H$-space structure. It gives us an element
\begin{equation*}
\tau_{2} \in \KO_{0}^{[2]}(\ZZ \times HGr, (0,x_{0}))
= \BO^{4,2}(\ZZ \times HGr, (0,x_{0}))
\end{equation*}
corresponding to the composition.  By Corollary \ref{C:p1} we have
\begin{equation*}
%\label{E:tau.formula.2}
\tau_{2} |_{\{i\} \times HGr(n,2n)}
= b_{1}(\mathcal U_{HGr(n,2n)},\phi_{HGr(n,2n)})+i\hh.
\end{equation*}
For any symplectically oriented cohomology theory $A^{*,*}$ we have
\cite[(9.3)]{Panin:2010aa}
\begin{equation*}
%\label{E:BO[[p]]}
A^{*,*}(\ZZ \times HGr) =
%\prod_{m \in \ZZ}
\bigl( A^{*,*}(\pt)[[b_{1},b_{2},b_{3},\dots]]^{\homog} \bigr)^{\times \ZZ}.
\end{equation*}
For such a theory let
\begin{align*}
\tfrac{1}{2}\rk^{A} & = (i 1_{HGr})_{i \in \ZZ}
\in A^{0,0}(\ZZ \times HGr),
&
b_{1}^{A} & = (b_{1})_{i \in \ZZ}\in A^{4,2}(\ZZ \times HGr)
\end{align*}
%
%be, respectively, the element which is $i \, 1_{HGr}$ on $\{i\} \times HGr$ for all $i$, and the
%element which is $b_{i}$
%on all the $\{i\} \times HGr$.
Then
\(
\tau_{2} = b_{1}^{\BO} + \tfrac 12 \rk^{\BO} \hh .
\)
Consider the element
\[
s = b^{\thy{U}}_1 \otimes 1_{\BO} + \tfrac{1}{2}\rk^{\thy{U}} \otimes \hh  \in
\bar {\thy{U}}^{4,2}(\ZZ \times HGr).
\]
Clearly one has $\bar\varphi(s) = \tau_{2}$.
The element $s$ may be regarded as a morphism of functors
$Hom_{H_{\bullet}(S)}(- , \ZZ \times HGr) \to \bar {\thy{U}}^{4,2}(-)$
by the Yoneda lemma.
The composite map
\[
Hom_{H_{\bullet}(S)}({-} , \ZZ \times HGr) \xra{s}
\bar {\thy{U}}^{4,2}(-)  \xra{\bar \varphi} \BO^{4,2}(-)
\]
coincides with a functor transformation given by the adjoint
$\Sigma_{T}^{\infty}(\ZZ \times HGr)(-2) \to \BO$
of the
motivic
weak equivalence
$\tau_2 \colon \ZZ \times HGr \to \KO^{[2]}$.
Thus for every pointed motivic space $X$ the map
\[
s_{X}\colon \BO^{4,2}(X)= Hom_{H_{\bullet}(S)}(X, \KO^{[2]})=
%\]
%%[\Sigma_{T}^{\infty}(X), \thy{U}]=
%\[ =
Hom_{H_{\bullet}(S)}(X, \ZZ \times HGr) \xra{s}   \bar {\thy{U}}^{4,2}(X) \]
is a section of the map
$\bar \varphi_{X}\colon \bar {\thy{U}}^{4,2}(X) \to \BO^{4,2}(X)$
which
is natural in $X$.
\end{proof}

\begin{lem}
\label{L:2}
For any integer $i$ the functorial map
$\bar\varphi_{X} \colon \bar{\thy{U}}^{8i+4,4i+2}(X) \to \BO^{8i+4,4i+2}(X)$ has a section
$t_{X}$ which is functorial in $X$.
\end{lem}

\begin{proof}
We have
\(
\BO^{8* + 4, 4* +2}=\BO^{4,2}[\beta_8,\beta^{-1}_8]
\)
for the periodicity element
$\beta_8 \in \BO^{8,4}(\pt)$ of Definition \ref{D:periodicity}.
So any element of
$\BO^{8* + 4, 4* +2}(X)$
may be written uniquely in the form  $a \cup \beta^i_8$ with
$a \in \BO^{4,2}(X)$ and $i \in \ZZ$.  We define
%\(
%%\label{Salgebraic}
%t_{X}\colon \BO^{8* + 4, 4* +2} \to \bar{\thy{U}}^{8* + 4, 4* +2}
%\)
%by
\[
t_{X}(a \cup \beta^i_8)= s_{X}(a)  \cup (1_{\thy{U}}\otimes \beta^i_8 )
\in \bar{\thy{U}}^{8* + 4, 4* +2}(X).
\]
%where
%$a \in \BO^{4,2}(X)$.
Then  $t_{X}$
is a section of
$\bar \varphi_{X}$ which is natural in $X$.
\end{proof}

\begin{lem}
\label{L:3}
If $X$ is a small pointed motivic space and $i$ is an integer,
then for any $\alpha \in \bar{\thy{U}}^{4i,2i}(X)$
there exists an $n \geq 0$
%with $n +i $ odd
with
$t_{X \wedge T^{\wedge 2n}} \circ\bar\varphi_{X\wedge T^{\wedge 2n}}(\Sigma_{T}^{2n}\alpha) =
\Sigma_{T}^{2n}\alpha$.
\end{lem}

\begin{proof}
We may assume that
$\alpha = a \otimes b$ with $a \in \thy{U}^{4d,2d}(X)$ and
$b \in \BO^{4i-4d,2i-2d}(\pt)$.
For a
small motivic space $X$
there is a canonical isomorphism \cite[Theorem 5.2]{Voevodsky:1998kx}
\begin{equation*}
%\label{IndLimit}
\thy{U}^{4d,2d}(X)= \colim_m
Hom_{H_{\bullet}(S)}(X \wedge T^{\wedge m},
\thy U_{2d+m}).
\end{equation*}
%
%where
%$\Sigma^{2n,n}= \Sigma^n_{T}$.
This isomorphism implies that there exists an integer
$n \geq 0$ such that
$\Sigma^{2n}_{T}a= f^*[u_{2d+2n}]$
for an appropriate map
$f\colon X \wedge T^{\wedge 2n} \to \thy{U}_{2d+2n}$
in
$H_\bullet(S)$.  We may assume that $d+n \geq N$ and that $n+i$ is odd.

We have
$[u_{2d+2n}] \otimes b \in \bar{\thy{U}}^{4n+4i,2n+2i}(\thy{U}_{2d+2n})$.
By hypothesis
\[
\bar\varphi_{\thy{U}_{2d+2n}} \colon \bar{\thy{U}}^{4n+4i,2n+2i}(\thy{U}_{2d+2n}) \to
\BO^ {4n+4i,2n+2i} (\thy{U}_{2d+2n})
\]
is an isomorphism.  So its section $t_ {\thy{U}_{2d+2n}} $ is the inverse isomorphism.
Hence we have
\begin{equation*}
%\label{KeyRelationForU}
(t_ {\thy{U}_ {2d+2n}} \circ \bar\varphi_ {\thy{U}_ {2d+2n}})([u_ {2d+2n}] \otimes b)
= [u_ {2d+2n}] \otimes b.
%\in\bar{\thy{U}}^{8* + 4, 4* +2}(\thy{U}_{2n}).
\end{equation*}
Then by the functoriality of $\bar{\thy{U}}$, $t$ and $\bar\varphi$ we have
\[
%\begin{multline*}
\Sigma^{2n}_{T} \alpha
%= f^*[u_{2d+2n}] \otimes b
= f^{*}([u_{2d+2n}] \otimes b)
= f^{*} \circ t_{\thy{U}_{2d+2n}} \circ \bar\varphi_ {\thy{U}_{2d+2n}}
([u_{2d+2n}] \otimes b)
%\\
%=
%t_{X \wedge T^{\wedge 2n}} \circ\bar\varphi_{X\wedge T^{\wedge 2n}}
%(f^*[u_{2d+2n}] \otimes b)
=
t_{X \wedge T^{\wedge 2n}} \circ\bar\varphi_{X\wedge T^{\wedge 2n}}(\Sigma_{T}^{2n}\alpha).
%\tag*{}
\qedhere
\]
%\end{multline*}
\end{proof}

\begin{lem}
\label{L:4}
Suppose for some $(p,q)$ that the homomorphism
$\bar \varphi_{X} \colon \bar{\thy{U}}^{p,q}(X) \to \BO^{p,q}(X)$
is an isomorphism for all small pointed motivic spaces $X$.
Then the same holds for $(p-1,q)$ and $(p-1,q-1)$.
\end{lem}

\begin{proof}
For $(p-1,q)$ this is because the suspension $\Sigma_{S^{1}}$ induces isomorphisms
$\thy{U}^{p-1,q}(X) \cong \thy{U}^{p,q}(X \wedge S^{1})$ and similar isomorphisms for $\bar{\thy{U}}$
and $\BO$, and these are compatible with $\varphi$ and $\bar\varphi$.
For $(p-1,q-1)$ use the suspension $\Sigma_{\GG_{m}}$.
\end{proof}

\begin{proof}
[Proof of Theorem \ref{AlmostlyMain}]
First suppose $(p,q) = (8i+4,4i+2)$ for some $i$.  Then for any small motivic space $X$ the map
$\varphi_{X} \colon \bar{\thy{U}}^{8i+4,4i+2}(X) \to \BO^ {8i+4,4i+2}(X)$ is surjective because it
has the section $t_{X}$ of Lemma \ref{L:2}.  To show it injective, we suppose $\alpha$ is in its kernel.
The suspension $\Sigma_{T}$ is compatible with $\varphi$ and $\bar\varphi$, so we have
$\bar\varphi_{X\wedge T^{\wedge 2n}}(\Sigma_{T}^{2n}\alpha)
= \Sigma_{T}^{2n}\varphi_{X}(\alpha) = 0$.
By Lemma \ref{L:3} we therefore also have $\Sigma_{T}^{2n}\alpha = 0$.
But $\Sigma_{T}^{2n}$ induces an isomorphism of cohomology groups.
%\colon \bar{\thy{U}}^{8i+4,4i+2}(X) \cong
%\bar{\thy{U}}^{8i+4n+4,4i+2n+2}(X \wedge T^{\wedge n})$ is an isomorphism.
So we have $\alpha = 0$.  Thus
$\bar \varphi_{X} \colon \bar{\thy{U}}^{p,q}(X) \to \BO^{p,q}(X)$ is an isomorphism
for all small motivic spaces $X$
for $(p,q) = (8i+4,4i+2)$.

The result for other values of $(p,q)$ follows from Lemma \ref{L:4} and a numerical argument.
\end{proof}

%%%%%%%%%%%%%%%%%%%%%%%%%%%%%%%%%%%%%%%%%%%%%%%%%%%%%%%%%%%%%%%%%%%%%
%\section{Main Result}
\section{Last details}

\begin{proof}
[Proof of Theorem \ref{T:main.CF}]
%We claim that we are in the setup of Theorem \ref{AlmostlyMain} with
%$(\thy{U},\thy{u}) = (\MSp,\thom^{\MSp})$ for the standard symplectic Thom orientation of $\MSp$
%of Example \ref{Ex:thom.MSp}.

By the
%Universality Theorem \ref{T:main.Sp}
universality of the symplectically oriented commutative
ring $T$-spectrum $(\MSp,\thom^{\MSp})$
(Theorem \ref{T:main.Sp})
there is a unique morphism $\varphi \colon \MSp \to \BO$ of commutative ring $T$-spectra
with $\varphi(\thom^{\MSp}) = \thom^{\BO}$.  It induces the morphisms
of \eqref {E:CF.hom}:
\begin{gather*}
\bar \varphi_{X}\colon \MSp^{\ast,\ast}(X) \otimes_{\MSp^{4*,2*}(\pt)}
\BO^{4*,2*}(\pt) \to
\BO^{\ast,\ast}(X),
\\
\bar \varphi_{X}\colon \MSp^{4*,2*}(X) \otimes_{\MSp^{4*,2*}(\pt)}
\BO^{4*,2*}(\pt) \to
\BO^{4*,2*}(X).
\end{gather*}
The second morphism, with the bidegrees $(4i,2i)$ only, is an isomorphism for $X = \MSp_{2r}$ for all $r$
by Theorem \ref{T:cellularity}.  So all the hypotheses of Theorem \ref{AlmostlyMain} hold
with $(\thy{U},\thy{u}) = (\MSp,\thom^{\MSp})$.
The conclusions of Theorem \ref{AlmostlyMain} imply Theorem \ref{T:main.CF}.
\end{proof}

%\bibliographystyle{siam}
%\bibliography{../HPn}

\end{document}